\documentstyle[twoside, amsmath, amssymb, amsfonts]{article}

\newtheorem{theorem}{Theorem}[section]
\newtheorem{lemma}[theorem]{Lemma}
\newtheorem{remark}[theorem]{Remark}
\newtheorem{proposition}[theorem]{Proposition}
\newtheorem{corollary}[theorem]{Corollary}
\newtheorem{definition}[theorem]{Definition}

\newenvironment{proof}{\begin{trivlist} \item[]{\em Proof.}}{\end{trivlist}}
\newcount\refno
\newcommand\be{\begin{equation}}
\newcommand\ee{\end{equation}}
\newcommand\bn{\begin{eqnarray}}
\newcommand\en{\end{eqnarray}}
\newcommand\bns{\begin{eqnarray*}}
\newcommand\ens{\end{eqnarray*}}
\newcommand\bd{\begin{definition}}
\newcommand\ed{\end{definition}}
\newcommand\br{\begin{remark}}
\newcommand\er{\end{remark}}
\newcommand\bt{\begin{theorem}}
\newcommand\et{\end{theorem}}
\newcommand\bp{\begin{proposition}}
\newcommand\ep{\end{proposition}}
\newcommand\bc{\begin{corollary}}
\newcommand\ec{\end{corollary}}
\newcommand\bl{\begin{lemma}}
\newcommand\el{\end{lemma}}
\newcommand\pf{\begin{proof}}
\newcommand\qed{\end{proof}\eop}

\newcommand\bC{{\mathbb C}}
\newcommand\bK{{\mathbb K}}
\newcommand\bR{{\mathbb R}}
\newcommand\bN{{\mathbb N}}

\newcommand\cF{{\cal F}}
\newcommand\cR{{\cal R}}

\newcommand{\NS}{\mbox{\scriptsize${\mathbb{N}}$}}

\newcommand{\F}{\mbox{$\mathcal{F}$}}

\def\eop{\hfill\rule{2.0mm}{2.0mm}}

\makeindex
\begin{document} 

%%%%%%%%%%%%%%%%%%%%%%%%%%%%%%%%%%%%%%%%%%%%%%%%%%%%%%%%%%%%%%%%%%%%%%
\title{$A$-sequences, $Z$-sequence, and $B$-sequences of Riordan Matrices}

\author{Tian-Xiao He\\
{\small Department of Mathematics}\\
 {\small Illinois Wesleyan University, Bloomington, Illinois 61702, USA}\\
}

\date{}

%\begin{document}

%%%%%%%%%%%%%%%%%%%%%%%%%%%%%%%%%%%%%%%%%%%%%%%%%%%%%%%%%%%%%%%%%%%%%%
\maketitle
\setcounter{page}{1}
\pagestyle{myheadings}
\markboth{Tian-Xiao  He}
{$A$-, $Z$-, and $B$-sequence of Riordan matrices } 

%%%%%%%%%%%%%%%%%%%%%%%%%%%%%%%%%%%%%%%%%%%%%%%%%%%%%%%%%%%%%%%%%%%%%%
%INSERT ABSTRACT HERE
\begin{abstract}
\noindent 
We defined two type $B$-sequences of Riordan arrays and present the $A$-sequence characterization and $Z$-sequence characterization of the Riordan matrices with two type $B$-sequences. 
The subgroups characterized by $A$-sequences and $Z$-sequences are studied. The application of the sequence characterization to the RNA type matrices is discussed. 
Finally, we investigate the $A$-, $Z$-, and $B$-sequences of the Pascal like Riordan matrices.


\end{abstract}

\section{Introduction}

Riordan matrices are infinite, lower triangular matrices defined by the generating function of their columns. They form a group, called {\em the Riordan group} (see Shapiro, Getu, W. J. Woan and L. Woodson \cite{SGWW}).

More formally, let us consider the set of formal power series $\F = K[\![$$t$$]\!]$, where ${\bK}$ is the field of ${\bR}$ or ${\bC}$. The \emph{order} of $f(t)  \in \F$, $f(t) =\sum_{k=0}^\infty f_kt^k$ ($f_k\in {\bK}$), is the minimum number $r\in\bN$ such that $f_r \neq 0$; $\F_r$ is the set of formal power series of order $r$. Let $g(t) \in \F_0$ and $f(t) \in \F_1$; the pair $(g,\,f )$ defines the {\em (proper) Riordan matrix} $D=(d_{n,k})_{n,k\in \NS}=(g, f)$ having
  
\begin{equation}\label{Radef}
d_{n,k} = [t^n]g(t) f(t)^k
\end{equation}
or, in other words, having $g f^k$ as the generating function of the $k$th column of $(g,f)$. 
The {\it first fundamental theorem of Riordan matrices } means the action of the proper Riordan matrices on the formal power series presented by  

\[
(g(t), f(t)) h(t)=g(t) (h\circ f)(t),
\]
which can be simplified to $(g,f)h=gh(f)$. Thus we immediately see that the usual row-by-column product of two Riordan matrices is also a Riordan matrix:
\begin{equation}\label{Proddef}
    (g_1,\,f_1 )  (g_2,\,f_2 ) = (g_1 g_2(f_1),\,f_2(f_1)).
\end{equation}
The Riordan matrix $I = (1,\,t)$ is the identity matrix because its entries are $d_{n,k} = [t^n]t^k=\delta_{n,k}$. Let $(g\left( t\right),\,f(t)) $ be a Riordan matrix. Then its inverse is

\begin{equation}
(g\left( t\right),\,f(t))^{-1}=\left( \frac{1}{g(\overline{{f}}(t))},
\overline{f}(t),\right)  \label{Invdef}
\end{equation}%
where $\overline {f}(t)$ is the compositional inverse of $f(t)$, i.e., $(f\circ 
\overline{f})(t)=(\overline{f}\circ f)(t)=t$. In this way, the set $\mathcal{
R}$ of all proper Riordan matrices forms a group (see \cite{SGWW}) called the Riordan group,

Here is a list of six important subgroups of the Riordan group (see \cite{Sha}).

\begin{itemize}
\item The {\it Appell subgroup} $\{ (g(t),\,t):g(t)\in {\cF}_0\}$.
\item The {\it Lagrange (associated) subgroup} $\{(1,\,f(t)):f(t)\in{\cF}_1\}$.
\item The {\it Bell subgroup} $\{(g(t),\, tg(t)):g(t)\in{\cF}_0\}$. 
\item The {\it hitting-time subgroup} $\{(tf'(t)/f(t),\, f(t)):f(t)\in{\cF}_1\}$.
\item The {\it derivative subgroup} $\{ (f'(t), \, f(t)):f(t)\in{\cF}_1\}.$
\item The {\it checkerboard subgroup} $\{ (g(t),\, f(t)):g(t)\in{\cF}_0\,\mbox{is even and}\, f(t)\in{\cF}_1\,\mbox{is odd}\}$. 
\end{itemize}

An infinite lower triangular matrix $[d_{n,k}]_{n,k\in{\bN}}$ is a Riordan matrix if and only if a unique sequence $A=(a_0\not= 0, a_1, a_2,\ldots)$ exists such that for every $n,k\in{\bN}$  
\be\label{eq:1.1}
d_{n+1,k+1} =a_0 d_{n,k}+a_1d_{n,k+1}+\cdots +a_nd_{n,n}. 
\ee 
This is equivalent to 
\be\label{eq:1.2}
f(t)=tA(f(t))\quad \text{or}\quad t=\bar f(t) A(t).
\ee
Here, $A(t)$ is the generating function of the $A$-sequence. The above first formula is also called the {\it second fundamental theorem of Riordan matrices } Moreover, there exists a unique sequence $Z=(z_0, z_1,z_2,\ldots)$ such that every element in column $0$ can be expressed as the linear combination 
\be\label{eq:1.3}
d_{n+1,0}=z_0 d_{n,0}+z_1d_{n,1}+\cdots +z_n d_{n,n},
\ee
or equivalently,
\be\label{eq:1.4}
g(t)=\frac{1}{1-tZ(f(t))},
\ee
in which and throughly we always assume $g(0)=g_0=1$, a usual hypothesis for proper Riordan matrices. From \eqref{1.4}, we may obtain 

\[
Z(t)=\frac{g(\bar f(t))-1}{\bar f(t)g(\bar f(t))}.
\]

$A$- and $Z$-sequence characterizations of Riordan matrices were introduced, developed, and/or studied in Merlini, Rogers, Sprugnoli, and Verri \cite{MRSV}, Roger \cite{Rog}, Sprugnoli and the author \cite{HS}, \cite{He15}, etc. In \cite{HS} the expressions of the $A$- and $Z$-sequences of the product depend on the analogous sequences of the two factors are given. More precisely, considering two proper Riordan matrices $D_1=(g_1,f_1)$ and $D_2=(g_2,f_2)$ and their product, 

\[
D_3=D_1 D_2=(g_1g_2(f_1), f_2(f_1)).
\]
Denote by $A_i(t)$ and $Z_i(t)$, $i=1,2,$ and $3$, the generating functions of $A$-sequences and $Z$-sequences of $D_i$, $i=1,2,$ and $3$, respectively. Then 

\be\label{1.7}
A_3(t)=A_2(t)A_1\left( \frac{t}{A_2(t)}\right)
\ee
and 

\be\label{1.7-2}
Z_3(t)=\left( 1-\frac{t}{A_2(t)}Z_2(t)\right) Z_1\left( \frac{t}{A_2(t)}\right)+A_1\left( \frac{t}{A_2(t)}\right) Z_2(t).
\ee

Let $A(t)$ and $Z(t)$ be the generating functions of the $A$- and $Z$-sequences of a Riordan matrix $D=(g,f)$ and let us denote by $g^\ast(t)$, $f^\ast(t)$, $A^\ast(t)$ and $Z^\ast(t)$ the corresponding power series for the inverse $D^{-1}=(g^\ast, f^\ast)$ and its $A$-sequence and $Z$-sequence. We immediately observe that  $f^\ast(t) = \overline{f}(t)$. Now we have (see \cite{HS}) that the $A$-sequence and $Z$-sequence of the inverse Riordan matrix $D^{-1}$ are, respectively, 

\be\label{1.8}
A^\ast\left( \frac{t}{A(t)} \right)=  \frac{1}{A(t)}
\ee
and 

\be\label{1.8-2}
Z^\ast \left( \frac{t}{A(t)}\right)=\frac{Z(t)}{tZ(t)-A(t)}.
\ee

Since a Riordan matrix arising in a combinatorial context has non-negative entries, it can not be an involution. Hence, we consider the set of the pseudo-involutions of Riordan group ${\cR}$, which means the set of all $D \in {\cR}$  such that $MD$ (and $DM$) is an involution, where $M = (1,-t)$.

Cheon, Jin, Kim, and Shapiro \cite{CJKS} (see also in Burlachenko \cite{Bur} and Phulara and Shapiro \cite{PS}) shows that a Riordan matrix $(g,f)$ is a Bell type pseudo-involution, i.e., $f=zg$ and $(g,-f)^2=(1,t)$, if and only if there exists a $B$-sequence, $\tilde B=(\tilde b_0, \tilde b_1, \tilde b_2,\ldots)$, characterizing all entries of a Riordan matrix, which is defined by 

\be\label{6.5}
d_{n+1,k}= d_{n,k-1}+\sum_{j\geq 0}{\tilde b}_jd_{n-j,k+j}
\ee
for $k\geq 0$, where $d_{n,-1}=0$, $n\geq 0$. However, for non-Bell type Riordan matrices there might exist two type $B$-sequences 
$B=(b_0,b_1,b_2,\ldots)$ and $\hat B=(\hat b_0, \hat b_1,\hat b_2,\ldots)$ defined by 

\be\label{6.5-2}
d_{n+1,k}= d_{n,k-1}+\sum_{j\geq 0}{b}_jd_{n-j,k+j}
\ee
for $k\geq 1$, and 

\be\label{6.5-3}
d_{n+1,0}= \sum_{j\geq 0}{\hat b}_jd_{n-j,k+j},
\ee
respectively. The $B$-sequence defined for all entries, $d_{n,k}$, $k\geq 1$, of a Riordan matrix not in the first column is called the {\it type-I  $B$-sequence}. The $B$-sequence defined for all entries, $d_{n,0}$, of the first column is called the {\it type-II $B$-sequence}.   We will show that for a Bell type Riordan matrix, it either has no $B$-sequence or has two type $B$-sequences and they are the same. However, there exist non-Bell type Riordan matrices which may have only one type $B$-sequences, which existence and construction are determined by their $A$-sequences and $Z$-sequence, respectively. More precisely, the existence and construction of their type-I $B$-sequences are characterized by the $A$-sequences of the Riordan matrices, while the existence and construction of the type-II $B$-sequences are characterized by the $Z$-sequences of the Riordan matrices.

This paper is devoted to the $A$-sequence and $Z$-sequence characterization of a Riordan matrix possessing $B$-sequence and the $A$-sequence and $Z$-sequence characterization of some subgroups of ${\cR}$. In next section, we discuss $A$-sequence characterization of the Riordan matrices possessing type-I $B$-sequences. Section $3$ presents $Z$-sequence characterization of the Riordan matrices possessing  type-II $B$-sequences. In Section $4$, we show some subgroups characterized by $A$-sequences, $Z$-sequence, and/or $B$-sequences. In the last section, Section $5$, we investigate the $A$-, $Z$-, and $B$-sequences of the Pascal like Riordan matrices.

\section{$A$-sequences and type-I $B$-sequences of Riordan matrices }

We now consider the $A$-sequence characterization of the existence of the type-I $B$-sequence for a Riordan matrix? which may not be a Bell type Riordan matrix. Here, the type-I $B$-sequence is  defined by \eqref{6.5-2}.

\begin{proposition}\label{pro:2.1}
Let $(g,f)=(d_{n,k})_{n,k\geq 0}$ be a Riordan matrix with a type-I $B$-sequence satisfying \eqref{6.5-2}, and let 
$A(t)$ and $B(t)$ be the generating functions of the $A$-sequence and the $B$-sequence of the Riordan matrix, respectively. Then we have the following equivalent formulas:

\bn
&&f=t+tfB(tf),\label{1.2}\\
&&t=\bar f +t\bar f B(t\bar f),\label{1.4}\\
&&A(t)=1 +tB(t^2/A(t)).\label{1.5}
\en
\end{proposition}

\begin{proof} 
Equation \eqref{6.5-2} can be written as 

\[
[t^{n+1}]gf^k= [t^n]gf^{k-1}+\sum_{j\geq 0}b_j [t^{n-j}]gf^{k+j},
\]
which implies \eqref{1.2}. Substituting $t=\bar f$, the compositional inverse of $f$ into \eqref{1.2}, we obtain \eqref{1.4}.The second fundamental theorem of Riordan matrix gives $\bar f =t/A$, which can be used to re-write equation \eqref{1.4} as \eqref{1.5}. 
\end{proof}
\eop

From Proposition \ref{pro:2.1}, we have the following result.

\begin{theorem}\label{thm:1.1}
Let $A=\sum_{j\geq 0}a_jt^j$ be the generating function of the $A$-sequence, $(a_j)_{j\geq 0}$, of a Riordan matrix $(g,f)$ that possesses a type-I $B$-sequence $( b_0, b_1, b_2,\ldots)$ defined by \eqref{6.5-2}, and let $f(t)=\sum_{j\geq 0} f_j t^j$. Then $a_0=1$ and $a_2=0$, or equivalently, 
$f_1=1$ and $f_3=f^2_2$.
\end{theorem}

\begin{proof}
From \eqref{1.5} and noting $a_0\not= 0$, it is easy to get  

\[
a_0=1.
\]
Denote $1/A(t)$ by $\hat C(t)$. Since 

\[
\hat C(t)=\frac{1}{A(t)}=c_0+\sum_{j\geq 1}c_jt^j,
\]
where $c_0=1$ and for $j\geq 1$ 
\be\label{1.6-1}
c_j=-\sum_{k\geq 1}a_kc_{j-k}.
\ee
Thus, we may solve $c_j$ from the above equations and substitute them to $\hat C(t)$ to obtain 
\be\label{1.6-0}
\hat C(t)=\frac{1}{A(t)}=1-a_1t+(a_1^2-a_2)t^2+(2a_1a_2-a_1^3-a_3)t^3+\cdots.
\ee
Comparing the coefficients of the powers of $t$ on the both sides of \eqref{1.5},
\be\label{1.6}
1+\sum_{j\geq 1}a_j t^j=1+b_0 t+b_1t^3\hat C(t)+b_2t^5\hat C(t)^2+b_3t^7\hat C(t)^3+\cdots,
\ee
we obtain the following system
\bns
&&a_1=b_0,\\
&&a_2=0,\\
&&a_3=b_1,\\
&&a_4=b_1c_1=-a_1b_1,\\
&&\cdots .
\ens
On the other hand, from the second fundamental theorem of Riordan matrices, $f=tA(f)$, we have 
\bns
&&f_1t+f_2t^2+f_3t^3+\cdots\\
&=&a_0t+a_1t(f_1t+f_2t^2+f_3t^3+\cdots)+a_2t^2(f_1t+f_2t^2+f_3t^3+\cdots)^2+\cdots.
\ens
Thus,  

\[
f_1=a_0=1, f_2=a_1f_1, f_3=a_1f_2,\ldots,
\]
which imply $f_1=1$ and $f_3=f_2^2$. 
\end{proof}
\eop

Theorem \ref{thm:1.1} gives a necessary condition for the existence of type-I $B$-sequence of a Riordan matrix. Now we establish a necessary and sufficient condition for the existence of type-I $B$-sequence of a Riordan matrix and its computation. Denote by $D_{n,m,k}$ the set of 
\be\label{eq:2.1+2}
D_{n,m,k}=\{{\mathbf {i}}=(i_{1},i_{2},\dots ,i_{k})\,:\ i_{1}+i_{2}+\cdots +i_{k}=n, i_1,i_2,\ldots,i_k\not=0\}
\ee
for $1\leq k\leq m\leq n$, where $k$ is the length of ${\mathbf{i}}$. Then the set 
\bn\label{eq:2.1+0}
D_{n,m}&&=\cup^m_{k=1} D_{n,m,k}.
\en
is the set of compositions of $n$ with the number of parts $k=1,2,\ldots, m$. If $m=n$, we write $D_{n,n}$ as $D_n$, namely, 
\be\label{eq:2.1+1}
{\mathcal {D}}_{n}=\{{\mathbf{i}}=(i_{1},i_{2},\dots ,i_{k})\,:\ 1\leq k\leq n,\ i_{1}+i_{2}+\cdots +i_{k}=n, i_1, i_2,\ldots, i_k\not=0 \},
\ee
\begin{theorem}\label{thm:1.1-2}
Let $(g,f)$ be a Riordan matrix, and let $A(t)=\sum_{j\geq 0} a_jt^j$ be the generating function of the $A$-sequence of $(g,f)$. Denote $\tilde c_j=c_{j-2}$ for $j\geq 2,$  where $c_j$ for $j\geq 0$ are shown in \eqref{1.6-1}, $\tilde c_{0}=0$, and $\tilde c_1=0$. Then $(g,f)$ has a $B$-sequence $(b_j)_{j\geq 0}$ defined by \eqref{6.5-2}, if and only if $A$-sequence of $(g,f)$ satisfies $a_0=1$, $a_2=0$ and for $\ell\geq 2$ 

\be\label{eq:2.1}
a_{2\ell}=\sum _{\mathbf {i} \in {\mathcal {D}}_{2\ell -1,\ell-1}}b_{k}\tilde c_{i_{1}}\tilde c_{i_{2}}\cdots \tilde c_{i_{k}},
\ee
where the index set, following the notation \eqref{eq:2.1+2}, is 

\bns
{\mathcal {D}}_{2\ell -1,\ell-1}&&=\cup^{\ell-1}_{k=1}{\mathcal {D}}_{2\ell-1,\ell-1,k},
\ens
where $D_{2\ell -1, \ell-1, k}$ are defined by \eqref{eq:2.1+2}. The right-hand side of equation \eqref{eq:2.1} is a function of $b_{j}$ for $1\leq j\leq \ell -1$ and $\tilde c_j$ for $0\leq j\leq 2\ell -1$ (or equivalently, $a_{j}$ for $0\leq j\leq 2\ell -3$). Here, for $\ell \geq 0$

\be\label{eq:2.2}
b_\ell=a_{2\ell+1}-\sum _{\mathbf {i} \in {\mathcal {D}}_{2\ell,\ell -1}}b_{k}\tilde c_{i_{1}}\tilde c_{i_{2}}\cdots \tilde c_{i_{k}},
\ee
where ${\mathcal {D}}_{0,-1}={\mathcal {D}}_{2,0}=\phi$, and for $\ell\geq 2$

\bns
{\mathcal {D}}_{2\ell,\ell-1 }&&=\cup^{\ell-1}_{k=1}{\mathcal D}_{2\ell, \ell-1,k}.\\
\ens
The summation of the right-hand side of \eqref{eq:2.2} is a function of $b_j$ for $1\leq j\leq \ell-1$ and $\tilde c_j$ for $0\leq j\leq 2\ell$ (or equivalently, $a_{j}$ for $0\leq j\leq 2\ell -2$). 

Furthermore,  $B$-sequence $(b_1,b_2,b_3,\ldots)$ can be evaluated by using \eqref{eq:2.2}, where $a_{2\ell +1}$, $\ell\geq 1$, are arbitrary.  Thus, we have 

\bns
&&a_0=1,\quad b_0=a_1, \quad a_2=0, \quad b_1=a_3,\\
&&a_4=b_1c_1=-b_1a_1=-a_1a_3,\\
&&b_2=a_5-b_1c_2=a_5-b_1(a_1^2-a_2),
\ens
etc. 
\end{theorem}

\begin{proof}
Consider the second term of the right-hand side of \eqref{1.5}. Let $B(t)=\sum _{n=0}^{\infty }{b_{n}}x^{n}$, and let $t^2\hat C(t)=\sum _{n=0}^{\infty }{\tilde c_{n}}x^{n}$ be the formal power series with $\tilde c_{0}=0$, $\tilde c_1=0$ and $\tilde c_j=c_{j-2}$ for $j\geq 2.$ Then the composition $B\circ (t^2\hat C)$ is again a formal power series, which can be written as, by using the {\it Fa\'a di Bruno's formula},   

\be\label{0.0}
B(t^2\hat C(t))=\sum _{n=0}^{\infty }{d_{n}}t^{n},
\ee
where $d_0 = b_0$ and the other coefficient $d_n$ for $n \geq 1$ can be expressed as a sum over compositions of $n$ or as an equivalent sum over partitions of $n$. More precisely, 

\be\label{0.00-1}
d_{n}=\sum _{\mathbf {i} \in {\mathcal {D}}_{n}}b_{k}\tilde c_{i_{1}}\tilde c_{i_{2}}\cdots \tilde c_{i_{k}},
\ee
where ${\mathcal{D}}_n$ is defined by \eqref{eq:2.1+1}. 
\footnote{Noting that the Fa\'a di Bruno formula can be considered as an application of the first fundamental theorem of Riordan matrices from Comtet \cite{comtet}, Roman \cite{Roman82_1} and \cite{Roman84book} and Roman and Rota \cite{RomanRota78}.}

We now apply \eqref{0.0} to \eqref{1.5} and compare the coefficients of the same power terms on the both sides of \eqref{1.5} to obtain 

\be\label{0.-2}
a_0=1 \quad \mbox{and}\quad a_n=d_{n-1}
\ee
for $n\geq 1$, where $d_n$ are presented by \eqref{0.00-1}. Thus, we have $a_0=1$ and $a_n=d_n$. 
It is clearly, $a_1=b_0$, $a_2=d_1=b_1\tilde c_1=0$, $a_3=d_2=b_1\tilde c_2=b_1$, 

\[
a_4=d_3=b_1\tilde c_3=b_1c_1=-a_1b_1,\quad a_5=d_4=b_1\tilde c_4+b_2\tilde c_2^2=b_1(a_1^2-a_2)+b_2.
\]
In general, if $n=2\ell+1$, $\ell\geq 2$, there is 

\[
a_{2\ell +1}=d_{2\ell}=\sum _{\mathbf {i} \in {\mathcal {D}}_{2\ell }}b_{k}\tilde c_{i_{1}}\tilde c_{i_{2}}\cdots \tilde c_{i_{k}},
\]
where ${\mathcal {D}}_{2\ell}$ is defined by \eqref{eq:2.1+1} for $n=2\ell$. 

By using the pigeonhole principle, for every $\ell+1 \leq k\leq 2\ell $, $(i_1,i_2,\ldots, i_k)$ contains at least one component 
to be $1$, which implies that $\tilde c_{i_{1}}\tilde c_{i_{2}}\cdots \tilde c_{i_{k}}=0$. Thus, the summation over the index set ${\mathcal {D}}_{2\ell}$ can be reduced to the summation over the index set 

\[
{\mathcal D}_{2\ell, \ell}=\{{\mathbf {i}}=(i_{1},i_{2},\dots ,i_{k})\,:\ 1\leq k\leq \ell ,\ i_{1}+i_{2}+\cdots +i_{k}=2\ell, i_1,i_2,\ldots,i_k\not=0 \}
\]
and consequently, 

\[
a_{2\ell+1}=d_{2\ell}=b_{\ell}\tilde c_2^\ell+\sum_{\mathbf{i}\in {\mathcal D}_{2\ell , \ell-1}}b_{k}\tilde c_{i_{1}}\tilde c_{i_{2}}\cdots \tilde c_{i_{k}},
\]
which implies \eqref{eq:2.2} because $\tilde c_2=c_0=1/a_0=1$. 

If $n=2\ell$, then from \eqref{1.5} and \eqref{0.00-1} we obtain 

\[
a_{2\ell}=d_{2\ell -1}=\sum_{{\mathbf {i}}\in {\mathcal D}_{2\ell -1}}b_k\tilde c_{i_1}\tilde c_{i_2}\ldots \tilde c_{i_k}.
\]
By using pigeonhole principle, the above summation over the index set ${\mathcal D}_{2\ell -1}$ is reduced to the summation over the index set 

\[
{\mathcal {D}}_{2\ell -1,\ell-1}=\{{\mathbf {i}}=(i_{1},i_{2},\dots ,i_{k})\,:\ 1\leq k\leq \ell -1,\ i_{1}+i_{2}+\cdots +i_{k}=2\ell -1, i_1,i_2,\ldots,i_k\not=0\},
\]
which proves that \eqref{eq:2.1} is true. 

Conversely, if \eqref{eq:2.1} and \eqref{eq:2.2} hold, one may immediately derive \eqref{1.5}, i.e., the riordan matrix $(g,f)$ possessing the $A$-sequence has the $B$-sequence that can be constructed by using \eqref{eq:2.2}. 
\end{proof}
\eop
\medbreak

\noindent{\bf Example 2.1} 
Considering the matrix $R=((1-t)g(t)/(1-tg(t)),tg(t))$ (see Cameron and Nkwanta \cite{CN}), where 

\be\label{RNA}
g(t)= \frac{1-t+t^{2}-\sqrt{1-2t-t^2-2t^3+t^4}}{2t^{2}}.
\ee
We may write its first few entries as 

\[
R=\left[\begin{array}{cccccccc}
1 &  &  &  &  &  &  &\\ 
1 & 1 &  &  &  &  & &\\ 
2 & 2 & 1 &  &  &  &  &\\ 
5 & 4 & 3& 1 &  &  & &\cdots  \\ 
12 &10 & 7 & 4& 1 &  & & \\ 
29 & 25 & 18 & 11& 5 & 1 &&  \\ 
71 & 62 & 47 & 30& 16& 6 &  1&\\ 
&  &  & \cdots  &  &  & &
\end{array}
\right].
\]
We call $R$ an RNA type matrix because it is related to the RNA matrix $R^\ast$ shown in Example 2.2. It is easy to find that the $A$-sequence and type-I $B$-sequence of $R$ are 

\[
A=(1,1,0,1,-1,\ldots)\quad \mbox{and} \quad B=(1,1,1,1,1,\ldots),
\]
respectively, which satisfy 

\[
a_0=1,b_0=a_1=1, a_2=0, b_1=a_3=1, a_4=-a_1a_3=-1,\ldots.
\]
From the second fundamental theorem of Riordan matrices, we have 

\[
A(t)=\frac{1+t+t^2+\sqrt{1+2t-t^2+2t^3+t^4}}{2}.
\]
Thus, from \eqref{1.5} we obtain that the generating function of the $B$-sequence of $R$ is 

\[
B(t)=\frac{1}{1-t}.
\]
On the other hand, we have $A(t)=t^2g(-t)$ or $\bar f(t)=-f(-t)=tg(-t)$ because $f(t)=tg(t)$ and $A(t)=t\bar f(t)$. Moreover, noticing \eqref{1.5} and 
\[
\frac{A(t)-1}{t}=\frac{1}{1-t^2/A(t)},
\]
we can also explain why $B(t)=1/(1-t)$. In addition, from \eqref{1.2} and $B(t)=1/(1-t)$ we obtain an identity for $g(t)$ 
\[
g(t)=1+\frac{tg(t)}{1-t^2g(t)},
\]
or equivalently,
\be\label{RNA-2}
(1-t+t^2)g(t)=1+t^2g(t)^2.
\ee

\medbreak

\noindent{\bf Remark 2.1}
An alternative way to present $d_n$ shown in the proof of Theorem \ref{thm:1.1-2} is 
\be\label{0.00-2}
d_{n}=\sum _{k=1}^{n}a_{k}\sum _{\mathbf {\pi } \in {\mathcal {P}}_{n,k}}{\binom {k}{\pi _{1},\pi _{2},...,\pi _{n}}}\tilde c_{1}^{\pi _{1}}\tilde c_{2}^{\pi _{2}}\cdots \tilde c_{n}^{\pi _{n}},
\ee
where
\[
{\mathcal {P}}_{n,k}=\{(\pi _{1},\pi _{2},\dots ,\pi _{n})\,:\ \pi _{1}+\pi _{2}+\cdots +\pi _{n}=k,\ \pi _{1}\cdot 1+\pi _{2}\cdot 2+\cdots +\pi _{n}\cdot n=n\}
\]
is the set of partitions of n into k parts in frequency-of-parts form.

The first form shown in \eqref{0.00-1} is obtained by picking out the coefficient of $t^n$ in $(\tilde c_{1}t+\tilde c_{2}t^{2}+\cdots )^{k}$ by inspection, and the second form \eqref{0.00-2} is then obtained by collecting like terms, or alternatively, by applying the multinomial theorem.
\medbreak

Theorem \ref{thm:1.1-2} has an analogy based on the expression \eqref{1.2}.

\begin{theorem}\label{thm:1.1-3}
Let $(g,f)$ be a Riordan matrix, and let $f(t)=\sum_{j\geq 1} f_jt^j$. Denote $\tilde f_j=f_{j-1}$ for $j\geq 1,$  where $f_{-1}, f_0=0$, and $\tilde b_j=b_{j-1}$ for $j\geq 1$, where $b_{-1}=0$. Then $(g,f)$ has a type-I $B$-sequence defined by \eqref{6.5-2} if and only if $f_1=1$, $f_3=f^2_2$ and for $\ell\geq 2$ there are 
\bn\label{eq:2.1-2}
f_{2\ell +1}=\sum _{\mathbf {i} \in {\mathcal {D}}_{2\ell +1}}\tilde b_{k}\tilde f_{i_{1}}\tilde f_{i_{2}}\cdots \tilde f_{i_{k}}
=f_2f_{2\ell}+ \sum _{\mathbf {i} \in {\mathcal {D}}'_{2\ell +1,\ell}}\tilde b_{k}\tilde f_{i_{1}}\tilde f_{i_{2}}\cdots \tilde f_{i_{k}},
\en
where the index sets are 
\bns
{\mathcal {D}}_{2\ell +1}&&=\cup^{2\ell+1}_{k=1}{\mathcal D}_{2\ell+1,2\ell+1,k}
\ens
and
\bns
{\mathcal {D}'}_{2\ell +1,\ell}&&=\cup^\ell_{k=2}{\mathcal D}_{2\ell+1,\ell,k}.
\ens
The summation on the leftmost-hand side of equation \eqref{eq:2.1-2} is a function of $b_{j}$ for $1\leq j\leq \ell -1$ and $f_j$ for $1\leq j\leq 2\ell-1$. Here, for $\ell\geq 1$, 
\be\label{eq:2.2-2}
b_{\ell-1}=f_{2\ell }-\sum _{\mathbf {i} \in {\mathcal {D}}_{2\ell,\ell-1}}\tilde b_{k}\tilde f_{i_{1}}\tilde f_{i_{2}}\cdots \tilde f_{i_{k}},
\ee
where ${\mathcal {C}}_{2,0}=\phi$ and for $\ell\geq 2$, 

\bn\label{2.2-3}
{\mathcal {D}}_{2\ell,\ell-1 }&&=\cup^{\ell-1}_{k=1}{\mathcal D}_{2\ell, \ell-1,k}.
\en
The summation of the right-hand side of \eqref{eq:2.2} is a function of $b_j$ for $0\leq j\leq \ell-2$ and $f_j$ for $1\leq j\leq 2\ell-1$. 

Furthermore,  $B$-sequence $(b_1,b_2,b_3,\ldots)$ can be evaluated by using \eqref{eq:2.2-2}, where $f_{2\ell}$, $\ell\geq 1$, are arbitrary. Thus, we have 

\bns
&&f_0=0,\quad f_1=1,\quad b_0=f_2, \quad f_3=b_0f_2=b_0^2, \quad b_1=f_4-b_0^3,\\
&&f_5=b_0f_4+b_1(2f_1f_2)=b_0^4+3b_0b_1,\\
&&b_2=f_6-b_0f_5-b_1(2f_1f_3+f_2^2)=f_6-b_0^5-5b_0^2b_1,
\ens
etc. 
\end{theorem}

\begin{proof}
From \eqref{1.2} we have 

\be\label{1.3-2}
f=t+tfB(tf)=t+\sum_{j\geq 0}b_j(tf)^{j+1}=t+\sum_{j\geq 0}\tilde b_j(tf)^j,
\ee
where $\tilde b_j=b_{j-1}$ and $\tilde b_0=b_{-1}=0$, and we may write 

\[
tf=\sum_{j\geq 1} \tilde f_jt^j.
\]
By using the Fa\'a di Bruno's formula we have 

\be\label{0.0-2}
\sum_{j\geq 0}\tilde b_j(tf)^j=\sum _{n=0}^{\infty }{c_{n}}t^{n},
\ee
where $c_0 = \tilde b_0=0$ and the other coefficient $c_n$ for $n \geq 1$ can be expressed as a sum over compositions of $n$ or as an equivalent sum over partitions of $n$. More precisely, 

\be\label{0.00-3}
c_{n}=\sum _{\mathbf {i} \in {\mathcal {D}}_{n}}\tilde b_{k}\tilde f_{i_{1}}\tilde f_{i_{2}}\cdots \tilde f_{i_{k}},
\ee
where ${\mathcal {D}}_n$ is defined by \eqref{eq:2.1+1}. Particularly, when $n=2\ell$, in the index set ${\mathcal {D}}_{2\ell}$, if $\ell+1 \leq k\leq 2\ell$, then there is at least one component of $(i_1,i_2,\ldots, i_k)$ to be $1$. Thus, the corresponding $\tilde f_{i_1}\tilde f_{i_2}\ldots \tilde f_{i_k}=0$. Consequently, the summation in \eqref{0.00-3} over the index set ${\mathcal {D}}_{2\ell}$ is reduced to the summation over the index set 

\[
{\mathcal D}_{2\ell,\ell}= \{{\mathbf{i}}=(i_{1},i_{2},\dots ,i_{k})\,:\ 1\leq k\leq \ell, i_{1}+i_{2}+\cdots +i_{k}=2\ell, i_1, i_2,\ldots, i_k\not=0 \}.
\]
Then combining \eqref{1.3-2} and \eqref{0.0-2} yields 

\be\label{f2l}
f_{2\ell}=c_{2\ell}=\sum _{\mathbf {i} \in {\mathcal {D}}_{2\ell }}\tilde b_{k}\tilde f_{i_{1}}\tilde f_{i_{2}}\cdots \tilde f_{i_{k}}
=\tilde b_\ell (\tilde f_{2})^\ell+\sum _{\mathbf {i} \in {\mathcal {D}}_{2\ell , \ell-1}}\tilde b_{k}\tilde f_{i_{1}}\tilde f_{i_{2}}\cdots \tilde f_{i_{k}},
\ee
where ${\mathcal {D}}_{2\ell, \ell-1}$ is shown in \eqref{2.2-3}, which implies \eqref{eq:2.2-2}. 

If $n=2\ell +1$, then \eqref{eq:2.1+1} becomes 
\[
{\mathcal {D}}_{2\ell +1}= \{{\mathbf{i}}=(i_{1},i_{2},\dots ,i_{k})\,:\ 1\leq k\leq 2\ell+1,\ i_{1}+i_{2}+\cdots +i_{k}=2\ell+1, i_1, i_2,\ldots, i_k\not=0 \}. 
\]
In the above index set, if $k=1$, then $(i_1)=(2\ell+1)$. Meanwhile for $\ell +1\leq k \leq 2\ell+1$, $(i_1,i_2,\ldots, i_k)$ contains at least one component to be $1$, which devotes $\tilde f_{i_1}\tilde f_{i_2}\ldots \tilde f_{i_k}=0$. Thus, the summation in 

\be\label{f2l+1}
f_{2\ell+1}=c_{2\ell+1}=\sum _{\mathbf {i} \in {\mathcal {D}}_{2\ell +1}}\tilde b_{k}\tilde f_{i_{1}}\tilde f_{i_{2}}\cdots \tilde f_{i_{k}}
\ee
over the index set ${\mathcal {D}}_{2\ell +1}$ can be reduced to the summation over the index set 
\[
{\mathcal D}_{2\ell+1,\ell}=\{{\mathbf{i}}=(i_{1},i_{2},\dots ,i_{k})\,:\ 1\leq k\leq \ell,\ i_{1}+i_{2}+\cdots +i_{k}=2\ell+1, i_1, i_2,\ldots, i_k\not=0 \}.
\]
Consequently, from \eqref{f2l+1} we obtain 
\bns
f_{2\ell+1}
&&=\sum _{\mathbf {i} \in {\mathcal {D}}_{2\ell +1,\ell}}\tilde b_{k}\tilde f_{i_{1}}\tilde f_{i_{2}}\cdots \tilde f_{i_{k}}\\
&& =\tilde b_1\tilde f_{2\ell +1}+\sum _{\mathbf {i} \in {\mathcal {D}'}_{2\ell +1,\ell }}\tilde b_{k}\tilde f_{i_{1}}\tilde f_{i_{2}}\cdots \tilde f_{i_{k}}\\
&& =b_0 f_{2\ell}+\sum _{\mathbf {i} \in {\mathcal {D}'}_{2\ell +1,\ell}}\tilde b_{k}\tilde f_{i_{1}}\tilde f_{i_{2}}\cdots \tilde f_{i_{k}},
\ens
where $b_0=f_2$ based on \eqref{1.3-2}. Thus if $(g,f)$ has a $B$-sequence, then we must have \eqref{eq:2.1-2}. 

Conversely, if \eqref{eq:2.1-2} and \eqref{eq:2.2-2} hold, then the Riordan matrix $(g,f)$ possesses a $B$-sequence, which can be evaluated by using \eqref{eq:2.2-2}. 
\end{proof}
\eop

\noindent{\bf Example 2.2} Consider {\it RNA matrix} (see Nkwanta \cite{Nkw} and Cheon, Jin, Kim, and Shapiro \cite{CJKS})

\[
R^\ast=(g,f)=( g(t), tg(t)),
\]
where $g(t)$ is given by \eqref{RNA}. Cameron and Nkwanta \cite{CN} show that 

\[
R^\ast =C_0^{-1}PC_0,
\]
where 

\[
C_0=\left( C(t^2), tC(t^2)\right)=\left( \frac{1-\sqrt{1-4t^2}}{2t^2}, 
\frac{1-\sqrt{1-4t^2}}{2t}\right),
\]
and
\[
P=\left( \frac{1}{1-t}, \frac{t}{1-t}\right).
\]
Hence, 

\[
C_0^{-1}=\left( \frac{1}{1+t^2}, \frac{t}{1+t^2}\right).
\]
In addition, the RNA matrix $R$ shown in Example 2.1 is related to $R^\ast$ in the sense 

\[
R=A_0^{-1}R^\ast A_0=A_0^{-1}C_0^{-1}PC_0A,
\]
where 

\[
A_0=\left( \frac{1}{1-t}, t\right), \quad \mbox{and}\quad A_0^{-1}=(1-t,t).
\]

The RNA matrix $R^\ast$ begins 

\[
R^\ast =\left[\begin{array}{cccccccc}
1 &  &  &  &  &  &  &\\ 
1 & 1 &  &  &  &  &  &\\ 
1 & 2 & 1 &  &  &  &  &\\ 
2 & 3 & 3& 1 &  &  & &\cdots  \\ 
4 &6 & 6 & 4& 1 &  & & \\ 
8 & 13 & 13 & 10& 5 & 1 &&  \\ 
17 & 28 & 30 & 24& 15 & 6 &  1&\\ 
&  &  & \cdots  &  &  & &
\end{array}
\right] 
\]
The elements in the leftmost column are the number of possible RNA secondary structures on a chain of length $n$, while other elements of the matrix count such chain with $k$ vertices designated as the start of a yet to be complete link. It is easy to check that $a_{0}=1$ and $a_{2}=0$, and, as in Eample 2.1, we have the $B$-sequence $B=( 1,1,1,\ldots)$. 

Consider another RNA type matrix 

\[
R^{\ast\ast}=(d(t),h(t))=\left( d(t), tg(t)\right),
\]
where $g(t)$ is defined before, $h=tg$, and 

\[
d(t)=\frac{g(t)-1}{t}= \frac{1-t-t^{2}-\sqrt{1-2t-t^2-2t^3+t^4}}{2t^{3}}.
\]
We call $R^{\ast\ast}$ an RNA type matrix because the elements in the left hand column are the number of possible RNA secondary structures on a chain of length $n$ except $n=0$.  The matrix $R^{\ast\ast}$ begins (see Nkwanta \cite{Nkw}) 

\[
R^{\ast\ast}=\left[\begin{array}{cccccccc}
1 &  &  &  &  &  &  &\\ 
1 & 1 &  &  &  &  &  &\\ 
2 & 2 & 1 &  &  &  &  &\\ 
4& 4 & 3& 1 &  &  & &\cdots  \\ 
8 &9 & 7 & 4& 1 &  & & \\ 
17 & 20 & 17& 11& 5 & 1 &  &\\ 
37 & 45& 41 & 29& 16 & 6 &  1&\\ 
&  &  & \cdots  &  &  & &
\end{array}
\right] 
\]
It is easy to check that $a_{0}=1$ and $a_{2}=0$, and it has a $B$-sequence $B=(1,1,1,\ldots)$ for all elements except 
those in the first column. The $Z$-sequence of $R^{\ast\ast}$ is $Z=(1,1,0,\ldots)$, which presents 

\[
d_{n+1,0}=d_{n,0}+d_{n,1},
\]
where $d_{n,0}$ is the number of secondary structure for $n$ points, and 

\[
d_{n,1}=\sum^{n-2}_{k=0}d_{k,0}d_{n-k-1}.
\]
The above equation has an analogue of Catalan matrix (see, for example, Stanley \cite{Sta} and \cite{He12}). Hence, $R^{\ast\ast}$ has type-I $B$-sequence, but no type-II $B$-sequence. 
\medbreak

\section{$Z$-sequence and type-II $B$-sequence of Riordan matrices }
We say the sequence $\hat B=(\hat b_0,\hat b_1, \hat b_2,\ldots)$ is a type-II $B$-sequence of a Riordan matrix $(g,f)=(d_{n,k})_{n,k\geq 0}$ if it satisfies 
\be\label{6.4}
d_{n+1,0}=\sum_{j\geq 0}\hat b_jd_{n-j,j}
\ee
for $n\geq 0$, i.e., for all entries except the first one in the first column of $(g,f)$. Obviously, a Riordan matrix may have no any type $B$-sequence, or has only one type $B$-sequence, or has both type $B$-sequences which are different. For instance, RNA matrix $R$ shown in Example 2.1 has type-I $B$-sequence $(1,1,1,\ldots)$, but no type-II $B$-sequence.  From the definition of type-II $B$-sequence, we immediately learn that its existence can be characterized by the $Z$-sequence of the Riordan matrix.  Here are some equivalent forms of \eqref{6.4} related to the $Z$-sequence of the Riordan matrix possessing type-II $B$-sequence. 

\begin{proposition}\label{pro:6.2}
If a Riordan matrix $(g,f)$ possesses a type-II $B$-sequence defined by \eqref{6.4}, then we have 

\bn
&&g=1+tg\hat B(tf),\label{6.4-2}\\
&&Z(f)=\hat B(tf).\label{6.4-3}\\
&&Z(t)=\hat B(t\bar f(t)),\label{6.4-4}
\en
where $\hat B(t)$ is the generating function of the type-II $B$-sequence $\hat B=(\hat b_0, \hat b_1,\hat b_2,\ldots)$.
\end{proposition}

\begin{proof}
From \eqref{6.4}, we have 

\[
[t^{n+1}]g=\sum_{j\geq 0} \hat b_j [t^{n-j}]gf^j=[t^{n+1}]tg\sum_{j\geq 0}\hat b_j(tf)^j
\]
for $n\geq 0$. Hence, \eqref{6.4-2} holds. From \eqref{6.4-2} and \eqref{eq:1.4} and noticing $g_0=1$, we obtain 

\[
\hat B(tf)=\frac{g-1}{tg}=\frac{1-1/g}{t}=Z(f),
\]
or expression \eqref{6.4-3}. Equation \eqref{6.4-4} follows after a substitution $t=\bar f$ applied in \eqref{6.4-3}.
\end{proof}
\eop

From the definition of Bell type Riordan matrices $(g,f)=(g, tg)$, it seems that its two type $B$-sequences, if exist, should have certain relationship. It is indeed true. More precisely, we will have the following result. 

\begin{proposition}\label{pro:6.1}
Let $(g,f)$ be a Riordan matrix. Then it is a Bell type Riordan matrix, i.e., $f=tg$, if and only if it either has the same type-I and type-II $B$-sequence or has no $B$-sequence. 
\end{proposition}
\begin{proof}
The proposition statement can be written as the following equivalent form, which provides a possible way to prove the proposition: 
Let $(g,f)$ be a Riordan matrix with a type-I (or type-II) $B$-sequence, $B=(b_0,b_1,\cdots )$ (or $\hat B=(\hat b_0,\hat b_1,\cdots ))$, defined by \eqref{6.5-2} (or \eqref{6.4}). Then $(g,f)$ is a Bell type Riordan matrix, i.e., $f=tg$, if and only if $(g,f)$ possesses a type-II (or type-I) $B$-sequence $\hat B$ (or $B$) defined by \eqref{6.4} (or \eqref{6.5-2}) with $B=\hat B$. It is sufficient to prove the last statement by considering type-I $B$-sequence. The case of type-II $B$-sequence can be proved with a similar argument. Let $(g,f)=(g,tg)$ has a typr-I $B$-sequence. Then from \eqref{1.2}, $f=t+tfB(tf)$, and $f=tg$. Thus, 

\be\label{6.2}
g=1+tgB(t^2g).
\ee
From \eqref{6.2}, we obtain $g(0)=1$, and from Proposition \ref{pro:6.2}, we know type-II $B$-sequence $\hat B$ exists and $\hat B=B$. Since the two type $B$-sequences are the same, the $B$-sequence characterizes all the entries of the Riordan matrix $(g,tg)$. 

Conversely, suppose a Riordan matrix $(g,f)$ has a $B$-sequence satisfying \eqref{6.4} for its first column entries and \eqref{6.5-2} for its other entries. Then from \eqref{1.2} and \eqref{6.4-2} and noticing $B=\hat B$, we must have 
\[
B(tf)=\frac{f-t}{tf}=\frac{g-1}{tg},
\]
which implies $f=tg$, i.e., $(g,f)$ is a Bell Riordan matrix. 
\end{proof}
\eop

\medbreak
\noindent{\bf Example 3.1} The RNA type matrices and the RNA matrix presented in Examples 2.1 and 2.2 have the same type-I $B$-sequence $(1,1,1,\ldots)$, while $R$ and $R^{\ast\ast}$ have no type-II $B$-sequence and $R^\ast$ has type-II $B$-sequence $(1,1,1,\ldots)$. It is easy to see that the Riordan matrix $(1/(1-2t), t/(1-t))$ has two different types $B$-sequences, $B=(1,0,0,\ldots)$ and $\hat B=(2,0,0,\ldots)$. 
The Riordan matrix $(1/(1-2f-t^2f), f)$, where 
\[
f=\frac{1-t-\sqrt{1-2t+t^2-4t^3}}{2t^2},
\]
has the type-I $B$-sequence $B=(1,1,0,\ldots)$ and the type-II $B$-sequence $\hat B=(2,1,0,\ldots)$. 
\medbreak
As recalled in the Introduction, a Riordan matrix is determined by its $A$-sequence and $Z$-sequence. If a Riordan matrix has a type-I $B$-sequence, then the Riordan matrix is determined by the type-I $B$-sequence and the $Z$-sequence. We now find the $Z$-sequence characterization for type-II $B$-sequences of Riordan matrices, including the existence and computation of type-II $B$-sequences.

\begin{proposition}\label{pro:6.3}
Let $(g,f)$ with $g(0)=1$ be a Riordan matrix $(g,f)$ possessing a type-II $B$-sequence defined by \eqref{6.4}, and let $Z(t)=\sum_{n\geq 0} z_nt^n$ be the generating function of the $Z$-sequence of $(g,f)$. Then $Z'(0)=0$, i.e., $z_1=0$, or equivalently, $g_1^2=g_0g_2$.
\end{proposition}
\begin{proof}
From \eqref{6.4-3}, $(g,f)$ has a type-II $B$-sequence defined by \eqref{6.4} implies that $z_1=0$, or equivalently, 
\[
Z'(0)=z_1=0
\]
because comparing the coefficient of $t$ on the both sides of \eqref{6.4-3} yields $z_1f_1=0$ and $f_1\not=0$. From \eqref{eq:1.4}, we have 

\[
Z(f)=\frac{g-1}{tg}=\frac{1}{t}\left( 1-\frac{1}{g}\right).
\]
Hence, 

\[
z_0+z_1f+z_2 f^2+\cdots= t^{-1}\left( 1-\left( \frac{1}{g_0}-\frac{g_1}{g_0^2} t+\left( \frac{g_1^2}{g_0^3}-\frac{g_2}{g_0^2}\right)t^2+\cdots\right)\right).
\]
Then $z_1=0$ is equivalent to 

\[
\frac{g_1^2}{g_0^3}-\frac{g_2}{g_0^2}=0,
\]
which is $g_1^2=g_0g_2$.
\end{proof}

\begin{theorem}\label{thm:6.2}
Let $Z=\sum_{j\geq 0}z_jt^j$ be the generating function of the $Z$-sequence, $(z_j)_{j\geq 0}$, of a Riordan matrix $(g,f)$, and let $\bar f=\sum_{j\geq 1}\bar f_j t^j$ be the compositional inverse of $f$.  Then $(g,f)$ possesses a type-II $B$-sequence $\hat B=( \hat b_0, \hat b_1, \ldots)$ defined by \eqref{6.4}  if and only if $\hat b_0=z_0$, $z_1=0$, and for $\ell \geq 1$ 

\bn\label{eq:2.1-3}
&&z_{2\ell+1} =\sum _{\mathbf {i} \in {\mathcal {D}}_{2\ell +1,\ell}} \hat b_{k}\bar f_{i_{1}-1}\bar f_{i_{2}-1}\cdots \bar f_{i_{k}-1},
\en
where the index set is 

\[
{\mathcal {D}}_{2\ell +1,\ell}=\{{\mathbf {i}}=(i_{1},i_{2},\dots ,i_{k})\,:\ 1\leq k\leq \ell ,\ i_{1}+i_{2}+\cdots +i_{k}=2\ell +1,i_1,i_2,\ldots, i_k\not= 0\},
\]
and the summation on the left-hand side of equation \eqref{eq:2.1-3} is a function of $\hat b_{j}$ for $1\leq j\leq \ell -1$ and $\bar f_j$ for $0\leq j\leq 2\ell$. Here for $\ell\geq 1$ $\hat b_\ell$ satisfy 

\be\label{eq:2.2-3}
\hat b_{\ell}=f_1^\ell\left(z_{2\ell }-\sum _{\mathbf {i} \in {\mathcal {D}}_{2\ell, \ell-1}} {\hat b}_{k}\bar f_{i_{1}-1}\bar f_{i_{2}-1}\cdots \bar f_{i_{k}-1}\right),
\ee
where ${\mathcal {D}}_{2,0}=\phi$ and for $\ell\geq 2$, 

\[
{\mathcal {D}}_{2\ell,\ell-1 }=\{{\mathbf {i}}=(i_{1},i_{2},\dots ,i_{k})\,:\ 1\leq k\leq \ell -1, \ i_{1}+i_{2}+\cdots +i_{k}=2\ell,i_1,i_2,\ldots, i_k\not= 0\}.
\]
The summation of the right-hand side of \eqref{eq:2.2} is a function of $\hat b_j$ for $0\leq j\leq \ell-1$ and $\bar f_j$ for $0\leq j\leq 2\ell-1$. 

Furthermore,  the type -II $B$-sequence $\hat B=(\hat b_1,\hat b_2,\hat b_3,\ldots)$ can be evaluated by using \eqref{eq:2.2-3}, where $z_{2\ell}$, $\ell \geq 1$, are arbitrary. Thus, we have 

\bns
&&\hat b_0=z_0,\quad z_1=0,\quad \hat b_1=f_1z_2,\\
&&z_3=\hat b_1\bar f_2, \quad \hat b_2=f_1^2\left( z_4-\hat b_1\bar f_3\right),\\
&&z_5=\hat b_1 \bar f_4+2\hat b_2\bar f_1\bar f_2,
\ens
etc.
\end{theorem}

\begin{proof} From \eqref{6.4-4}, $Z(t)=B(t\bar f)$, we use the Fa\' a di Bruno formula to its right-hand side and compare the coefficients on its both sides. Thus  

\be\label{6.6} 
z_n=\sum _{\mathbf {i} \in {\mathcal {D}}_{n}}\hat b_{k}\bar f_{i_{1}-1}\bar f_{i_{2}-1}\cdots \bar f_{i_{k}-1}=\sum _{\mathbf {i} \in {\mathcal {E}}_{n}}\hat b_{k}\hat {\bar f}_{i_{1}}\hat {\bar f}_{i_{2}}\cdots \hat {\bar f}_{i_{k}},
\ee
where $\hat {\bar f}_j=\bar f_{j-1}$, $\hat {\bar f}_0=\bar f_{-1}=0$, $\hat {\bar f}_1=\bar f_0=0$, $\bar f_n=[t^n] \overline{f}$, and ${\mathcal {D}}_n$ is defined by \eqref{eq:2.1+1}.  It is clear that $z_0=\hat b_0$, 
\[
z_1=\hat b_1\hat {\bar f}_1=0,
\]
and 
\[
z_2=\hat b_1\hat {\bar f}_2+\hat b_2\hat {\bar f}_1^2=\hat b_1\bar f_1,
\]
which yields 
\[
\hat b_1=\frac{z_2}{ \bar f_1}=f_1z_2
\]
because $1/\bar f_1=f_1$. In general, for $n=2\ell+1$ and $\ell\geq 1$, we have 
\bns
z_{2\ell+1}=\sum _{\mathbf {i} \in {\mathcal {D}}_{2\ell+1}}\hat b_{k}\hat {\bar f}_{i_{1}}\hat {\bar f}_{i_{2}}\cdots \hat {\bar f}_{i_{k}}=\sum _{\mathbf {i} \in {\mathcal {D}}_{2\ell+1,\ell}}\hat b_{k}\hat {\bar f}_{i_{1}}\hat {\bar f}_{i_{2}}\cdots \hat {\bar f}_{i_{k}},
\ens
where the last equation is due to the fact of $\hat {\bar f}_{i_{1}}\hat {\bar f}_{i_{2}}\cdots \hat {\bar f}_{i_{k}}$ contains 
at least one factor of $\hat {\bar f}_1=0$ for all $\ell+1\leq k\leq 2\ell+1$. Hence, we obtain \eqref{eq:2.1-3}. To determine 
$\hat b_j$, we substitute $n=2\ell$, $\ell\geq 1$, into \eqref{6.6} to have 

\bns
&&z_{2\ell }=\sum _{\mathbf {i} \in {\mathcal {D}}_{2\ell}}\hat b_{k}\hat {\bar f}_{i_{1}}\hat {\bar f}_{i_{2}}\cdots \hat {\bar f}_{i_{k}}=\hat b_\ell \hat{\bar f}_2^\ell +\sum _{\mathbf {i} \in {\mathcal {D}}_{2\ell,\ell-1}}\hat b_{k}\hat {\bar f}_{i_{1}}\hat {\bar f}_{i_{2}}\cdots \hat {\bar f}_{i_{k}},
\ens
where the last equation is due to the fact of $\hat {\bar f}_{i_{1}}\hat {\bar f}_{i_{2}}\cdots \hat {\bar f}_{i_{k}}$ contains 
at least one factor of $\hat {\bar f}_1=0$ for all $\ell+1\leq k\leq 2\ell$. From the last expression about $z_{2\ell}$ and noticing $\hat{\bar f}_2={\bar f}_1=1/f_1$, we obtain \eqref{eq:2.2-3}. 

Conversely, if \eqref{eq:2.1-3} and \eqref{eq:2.2-3} hold, one may derive \eqref{6.4-4}, i.e., the Riordan matrix $(g,f)$ possessing the $Z$-sequence has a type-II $B$-sequence, where the type-II $B$-sequence can be constructed by using \eqref{6.4}. 
\end{proof}
\eop

\section{Subgroups of Riordan group characterized by $A$- and $Z$-sequences}
We now discuss the subgroups of the Riordan group defined by the $A$-sequence and $Z$-sequence of Riordan matrices.

\begin{theorem}\label{thm:1.3}
The set of the Riordan matrices with $A$-sequences of the form $(1,a_1,0,a_3,\ldots)$, denoted by $R_{0,2}$, is a subgroup of the Riordan group. 
\end{theorem}

\begin{proof}
If $D_{1}$ and $D_{2}$ are in $R_{0,2}$, then the generating functions of their $A$-sequences are 

\[
A_1(t)=a_{1,0}+a_{1,1}t+a_{1,3}t^3+\cdots
\]
and 

\[
A_2(t)=a_{2,0}+a_{2,1}t+a_{2,3}t^3+\cdots.
\]
From \eqref{1.6-0}, and noting $a_{2,2}=0$, we have 

\[
\frac{t}{A_2(t)}=t\hat C(t)=t(1-a_{2,1}t+a_{2,1}^2t^2-(a_{2,1}^3+a_{2,3})t^3+\cdots.
\]
Thus the generating function $A_3(t)$ of the $A$-sequence of $D_3=D_1D_2$ is 

\bns
A_3(t)&=&A_2(t)A_1\left( \frac{t}{A_2(t)}\right)\\
&=&(a_{2,0}+a_{2,1}t+a_{2,3}t^3+\cdots)(a_{1,0}+a_{1,1}t(1-a_{2,1}t+a_{2,1}^2t^2-(a_{2,1}^3+a_{2,3})t^3
+\cdots)\\
&&\quad +a_{1,3}t^3(1-a_{2,1}t+a_{2,1}^2t^2-(a_{2,1}^3+a_{2,3})t^3+\cdots)^3\cdots\\
&=&a_{1,0}a_{2,0}+a_{1,0}a_{2,1}t+a_{1,1}(a_{2,0}+a_{2,1}t)(-a_{2,1}t+a_{2,1}^2t^2)+ct^3+\cdots\\
&=&1+a_{2,1}t+a_{1,1}(1+a_{2,1}t)(-a_{2,1}t+a_{2,1}^2t^2+ct^3+\cdots=1+a_{2,1}t+ct^3+\cdots,
\ens
which implies that $D_3\in R_{0,2}$. 

If $D\in R_{0,2}$, then the generating function, $A^\ast(t)$ of the $A$-sequence of the inverse $D^\ast$ of $D$ is 

\[
A^\ast =\frac{t}{f(t)}=\frac{t}{t+f_2t^2+f_2^2t^3+\cdots}=1-f_2t+c't^3+\cdots,
\]
which refers to that $D^\ast\in R_{0,2}$. The proof is complete.
\end{proof}

\medbreak

\noindent{\bf Remark 4.1} Luz\'on, Mor\'on, and Prieto-Martinez \cite{LMPM} show that all Riordan matrices with $A$-sequence $(a_0,a_1,0, a_3,\ldots)$ form a subgroup of the Riordan group. Hence, $R_{0,2}$ is a subgroup of their subgroup. 

\medbreak
\noindent{\bf Remark 4.2} Here is an alternative proof of Theorem \ref{thm:1.3}, which is based on a concept of truncation class of formal power series. 

More precisely, let $f=\sum_{j\geq 0} f_j t^j$ and $h=\sum_{j\geq 0}h_jt^j$ be two power series. If there exists an integer $r\geq 0$ such that the $r$-th truncations of $f|_r=\sum^r_{j=0} f_jt^j$ and $h|_r=\sum^r_{j=0}h_jt^j$ satisfying $f|_r\equiv ch|_r$ for some non-zero constant $c$ and $f_{r+1}\not= c h_{r+1}$ for any constant $c\not= 0$, then we say $f$ and $h$ have the same truncation of order $r$. For a fixed power series $f$ and an integer $r\geq 0$, the collection of all power series that possess the same truncation of order $r$ is called a truncation class of order $r$ with respect to $f$. This class is denoted by $T_r(f)$. 

Firstly, let $(g,f)$ be a Riordan matrix, where $f=\sum_{j\geq 1}f_jt^j$, $f_1\not=0$, and let $(a_0,a_1,a_2,\ldots)$ be the $A$-sequence of $(g,f)$. Then, $a_2=0$ and $a_{3}\not=0$  if and only if 

\be\label{3.1}
f_2^2=f_1f_{3}, 
\ee
or equivalently, the truncation of the first $3$ terms of $f$ can be written as

\be\label{3.2}
f|_r=a_0t\frac{1-(a_1t)^2}{1-a_1t},
\ee
i.e., $f$ is in $T_2(t/(1-a_1t))$. In fact, from the second fundamental theorem of Riordan matrices, $f=tA(f)$, where $f=\sum_{j\geq 1}f_jt^j$ with $f_1\not=0$ and $A(t)=\sum_{i\geq 0} a_it^i$, we have 

\bn\label{3.3}
&&\sum_{j\geq 0}f_jt^j=a_0t+a_1t\sum_{j\geq 1}f_jt^j +a_2t\left(\sum_{j\geq 1}f_jt^j\right)^2+a_3t\left(\sum_{j\geq 1}f_jt^j\right)^3+\cdots .
\en
Thus, 

\[
a_0=f_1\not=0\quad \mbox{and} \quad a_1=\frac{f_2}{f_1}.
\]
If $a_2=0$, then 

\[
f_{3}=a_1f_{2}.
\]
Consequently, 
\[
\frac{f_{3}}{f_{2}}=\frac{f_2}{f_1}=a_1,
\]
which implies $f_1f_{3}=f_2^2$,  $f_{3}=f_2^2/f_1$, and  

\[
f_{3}=a_1^2f_1.
\]
Hence, the first $3$ terms of $f$ is 

\[
\sum_{j=1}^2f_jt^j=f_1t\sum^{1}_{j=0}a_1^jt^j=f_1t\frac{1-(a_1t)^2}{1-a_1t}.
\]
In other word, $f$ is in $T_2(t/(1-a_1t))$ because 

\[
f|_2=\left. a_0\frac{t}{1-a_1t}\right|_2.
\]

Conversely, if \eqref{3.2}, or equivalently, \eqref{3.1} hold for $j=2$, then from \eqref{3.3} we have 

\[
a_1f_2=f_{3}=a_1 f_2+a_2f_1,
\]
which implies $a_2=0$ due to $f_1\not= 0$. Secondly, let $(g,f)$ be a Riordan matrix with $A$-sequence $(a_0, a_1,\ldots)$, and let $f|_2=a_0t/(1-a_1t)|_2$. Then the compositional inverse of $f$, $\bar f$, has the truncation of order $2$ of the form  

\be\label{3.4}
\bar f|_2=\left. \frac{t}{a_0+a_1t}\right|_2.
\ee
If $f|_2=a_0t/(1-a_1t)|_2$, then $\bar f$ is the compositional inverse of $f$ if and only if $\bar f|_2$ can be presented by \eqref{3.4} and $\bar f=t/(a_0+a_1t)$. In fact, $f\circ \bar f=t$ implies $(f\circ \bar f)|_2=t$. If $f$ has the truncation of the first $2$ terms presented by $f|_2=a_0t/(1-a_1t)|_2$, then a straightforward process can be applied to solve $\bar f|_2$ shown in \eqref{3.4} from the equation $(f\circ \bar f)|_2=t$. If $f|_2=a_0t/(1-a_1t)|_2$, then $\bar f$ possesses the truncation of order $2$ shown in \eqref{3.4}. Thus $\bar f=t/(a_0+a_1t)$. Conversely, if $\bar f$ satisfies \eqref{3.4}, i.e., $\bar f=t/(a_0+a_1t)$, then $f=a_0t/(1-a_1t)$, which implies $f|_2=a_0t/(1-a_1t)|_2$. 

Finally, we show that the set of Riordan matrices, denoted by $R_{0,2}$ with $A$-sequences, $(a_{0}, a_{1},$ $a_{2},\ldots)$, satisfying $a_{0}=1$ and $a_{2}=0$ forms a subgroup of the Riordan group.

Let $(g_1,f_1)$ and $(g_2,f_2)$ be two Riordan matrices with $A$-sequences $A_1$ and $A_2$, and let 
$(g_3,f_3)=(g_1,f_1)(g_2,f_2)$ with $A$-sequence $A_3$. From the Second Fundamental Theorem of Riordan Arrays, $f=tA(f)$, we have $t=\bar f A$. Thus we may rewrite \eqref{1.7} as 

\[
A_3(t)=A_2(t)A_1\left( \bar f_2\right).
\]
If $(g_1,f_1)$ and $(g_2,f_2)\in R_{0,2}$, then 

\[
A_i(t)=a_{i,0}+a_{i,1}t+a_{i,3}t^{3}+\cdots=a_{i,0}+a_{i,1}t+O(t^{3})
\]
for $i=1$ and $2$. Hence, 
\[
f_2|_2=\left. \frac{a_{2,0}t}{1-a_{2,1}t}\right|_2,
\]
which implies 

\[
\bar f_2|_2=\left.\frac{t}{a_{2,0}+a_{2,1}t}\right|_2.
\]
Combining the above equations yields

\bns
A_3(t)&=&(a_{2,0}+a_{2,1}t+O(t^{3}))\left(a_{1,0}+\left.a_{1,1}\frac{t}{a_0+a_1t}\right|_2+O(t^{3})\right)\\
&=&a_{1,0}a_{2,0}+a_{1,0}a_{2,1}t+\left.a_{1,1}(a_{2,0}+a_{2,1}t)\frac{t}{a_0+a_1t}\right|_r +O(t^{3})\\
&=&a_{1,0}a_{2,0}+(a_{1,0}a_{2,1}+a_{1,1})t+O(t^{3})=1+(a_{1,1}+a_{2,1})t+O(t^{3}),
\ens
which implies $(g_3,f_3)=(g_1,f_1)(g_2,f_2)$ is also in $R_{0,2}$, where we use the obvious result 

\[
a_{3,0}=A_3(0)=A_1(0)A_2(0)=1.
\]

\medbreak

\noindent{\bf Remark 4.3} We now give a more direct way to prove Theorem \ref{thm:1.3}. 
If $D_{1}$ and $D_{2}$ are in $R_{0,2}$, then the generating functions of their $A$-sequences satisfy 

\[
A_{1}(0)=A_{2}(0)=1\,\, \mbox{and}\,\, A''_{1}(0)=A''_{2}(0)=0.
\]
From \eqref{1.7}, we have 

\[
A_{3}(0)=A_{2}(0)A_{1}(0)=1
\]
and 

\bns
A_{3}''(t)&&=A''_{2}(t)A_{1}\left( \frac{t}{A_{2}(t)}\right) +2A'_{2}(t)A'_{1}\left( \frac{t}{A_{2}(t)}\right)\left( \frac{t}{A_{2}(t)}\right)'\\
&& +A_{2}(t)\left( A'_{1}(t)\left( \frac{t}{A_{2}(t)}\right)\left( \frac{t}{A_{2}(t)}\right)'\right)'\\
\ens
which implies 

\bns
A_{3}''(0)&&=2A'_{2}(0)A'_{1}(0)\frac{A_{2}(0)}{A^{2}_{2}(0)}+A_{2}(0)A''_{1}(0)\left( \frac{A_{2}(0)}{A^{2}_{2}(0)}\right)^{2}\\
&&+A_{2}(0)A'_{1}(0)( -2)A^{-2}_{2}(0)A'_{2}(0)=0
\ens
Hence, $R_{0,2}$ is closed under the Riordan multiplication. Similarly, we may use \eqref{1.8} to prove that the inverse, $D^{-1}$, of any element $D\in R_{0,2}$ with the $A$-sequence $(1,a_{1}, 0, a_{3},\ldots)$ has the $A^{\ast}$-sequence $(1, a^{\ast}_{1}, 0, a^{\ast}_{3},\ldots)$. Thus $D^{-1}\in R_{0,2}$. More precisely, from \eqref{1.8}, 

\[
A^\ast (0)=\frac{1}{A(0)}=1,
\]
which implies $a^\ast_0=a_0=1$. Taking derivatives on the both sides of \eqref{1.8} yields

\[
(A^\ast)'\left(\frac{t}{A(t)}\right)\frac{A(t)-tA'(t)}{A(t)^2}=-\frac{A'(t)}{A(t)^2}.
\]
Substituting $t=0$ and noting $A(0)=1$, we obtain 

\[
(A^\ast)'(0)=-A'(0).
\]
Taking second derivatives on the both sides of \eqref{1.8}, we have 

\bns
&&(A^\ast)''\left(\frac{t}{A(t)}\right)\left(\frac{A(t)-tA'(t)}{A(t)^2}\right)^2\\
&&+(A^\ast)'\left(\frac{t}{A(t)}\right)\frac{-tA''(t)A(t)^2-2A(t)A'(t)(A(t)-tA'(t))}{A(t)^4}
=\frac{2(A'(t))^2}{A(t)^3}-\frac{A''(t)}{A(t)^2}.
\ens
When $t=0$, one can derive 

\[
(A^\ast)''(0)-2A'(0)(A^\ast)'(0)=2A'(0)^2-A''(0).
\]
Since $(A^\ast)'(0)=-A'(0)$, the above equation can be reduced to 

\[
(A^\ast)''(0)=-A''(0)=0,
\]
which yields $a^\ast_2=0$.
\medbreak

\begin{proposition}\label{pro:1.4}
Let two proper Riordan matrices $D_1=(g_1,f_1)$ and $D_2=(g_2,f_2)$ have type-I $B$-sequences $B_1$ and $B_2$, respectively, then the product of $D_1$ and $D_2$, 

\[
D_3=D_1 D_2=(g_1g_2(f_1), f_2(f_1)),
\]
has type-I $B$-sequence $B_3$ with its generating function

\bns
&&B_3\left( \frac{t^2}{A_3(t)}\right) =B_2\left( \frac{t^2}{A_2(t)}\right)+\frac{1}{A_2(t)}B_1\left( \frac{t^2}{A_2(t)A_3(t)}\right)\\
&&\quad +\frac{t}{A_2(t)}B_2\left( \frac{t^2}{A_2(t)}\right) B_1\left( \frac{t^2}{A_2(t)A_3(t)}\right).
\ens
\end{proposition}
\begin{proof}
Substituting \eqref{1.5} into \eqref{1.7}, one may obtain the result.
\end{proof}

Combining Theorems \ref{thm:1.1} and \ref{thm:1.3} together, we immediately have the following result. 

\begin{theorem}\label{thm:1.4}
If a Riordan matrix $(g,f)$ has a type-I $B$-sequence, then it is in the subgroup $R_{0,2}$.  
\end{theorem}

\begin{theorem}\label{thm:6.4}
The set of Riordan matrices, denoted by $R_{1,1,1}$, with $A$-sequences of the form  $(1, a_1, a_2,\ldots)$ and $Z$-sequences of the form $(z_0=a_1,0,z_2,\ldots)$ forms a subgroup of the Riordan group.
\end{theorem}

\begin{proof}
Let $(g_1,f_1)$ and $(g_2,f_2)\in R_{1,1,1}$, and let $(g_3,f_3)=(g_1,f_1)(g_2,f_2)$, where the $A$-sequences $(a_{i,j})_{j=0,1,\ldots}$ and the $Z$-sequences $(z_{i,j})_{j=0,1,\ldots}$ of $(g_i,f_i)$, $i=1$ and $2$, satisfy the conditions

\[
a_{i,0}=1,\quad  a_{i,1}=z_{i,0}, \quad \mbox{and} \quad z_{i,1}=0,
\]
for $i=1$ and $2$. Then from \eqref{1.7} we have the generating function of the $A$-sequence of $(g_3,f_3)$  

\be\label{1.7-1}
A_3(t)=A_2(t)A_1\left( \frac{t}{A_2(t)}\right).
\ee
Hence from $A_i(0)=a_{i,0}=1$, the constant term of $A_3(t)$ is 

\[
A_3(0)=A_2(0)A_1(0)=1,
\]
i.e., $a_{3,0}=1$. Furthermore, \eqref{1.7-1} also implies 

\[
a_{3,1}=A'_3(0)=A'_2(0)A_1(0)+A_2(0)A'_1(0)=a_{1,1}+a_{2,1}.
\]

From \eqref{1.7-2}, we obtain the generating function of the $Z$-sequence of $(g_3,f_3)$ 

\be\label{1.7-2-1}
Z_3(t)=\left( 1-\frac{t}{A_2(t)}Z_2(t)\right) Z_1\left( \frac{t}{A_2(t)}\right)+A_1\left( \frac{t}{A_2(t)}\right) Z_2(t).
\ee
Therefore, the constant term of $Z_3(t)$ is 

\[
Z_3(0)=Z_1(0)+A_1(0)Z_2(0)=a_{1,1}+a_{2,1}=A'_3(0),
\]
or equivalently,

\[
z_{3,0}=a_{3,1}.
\]
In addition, \eqref{1.7-2-1} gives

\bns
Z'_3(0)&&=(-Z_2(0))Z_1(0)+A'_1(0)Z_2(0)+A_1(0)Z'_2(0)\\
&&=-z_{1,0}z_{2,0}+a_{1,1}z_{2,0}=-z_{1,0}z_{2,0}+z_{1,0}z_{2,0}=0.
\ens

Let $(g,f)\in R_{1,1,1}$, and let $(g^\ast,f^\ast)=(g,f)^{-1}$, i.e., $g^\ast=1/(g\circ \bar f)$ and $f^\ast=\bar f$, where the $A$-sequences $(a_j)_{j=0,1,\ldots}$ and the $Z$-sequences $(z_{j})_{j=0,1,\ldots}$ of $(g,f)$ satisfy the conditions

\[
a_{0}=1,\quad  a_{1}=z_{0}, \quad \mbox{and} \quad z_1=0.
\]
Then from \eqref{1.7} and \eqref{1.7-2}, we have 

\be\label{1.8-1}
A^\ast\left( \frac{t}{A(t)} \right)=  \frac{1}{A(t)}
\ee
and 

\be\label{1.8-2-1}
Z^\ast \left( \frac{t}{A(t)}\right)=\frac{Z(t)}{tZ(t)-A(t)},
\ee
respectively. Denote the $A$-sequence and the $Z$-sequence of $(g,f)^{-1}$ by $(a^\ast_j)_{j=0,1,\ldots}$ and 
$(Z^\ast_j)_{j=0,1,\ldots}$. Then \eqref{1.8-1} yields 

\[
a^\ast_0=A^\ast(0)=\frac{1}{A(0)}=1,
\]
and 

\[
a^\ast_1=(A^\ast(t))'|_{t=0}=-\frac{1}{A^2(0)}A'(0)=-a_1.
\]
Meanwhile, using \eqref{1.8-2-1} we obtain

\[
Z^\ast_0=Z^\ast(0)=-Z(0)=-z_0=-a_1=a^\ast_1.
\]
Finally, we have 

\bns
Z^\ast_1&&=(Z^\ast(t))'|_{t=0}=-\frac{Z'(0)A(0)-Z(0)(A'(0)-Z(0)}{A(0)^2}\\
&&=-\frac{Z(0)^2-A'(0)Z(0)}{A(0)^2}=a_1z_0-z_0^2=0
\ens
because $a_1=z_0$. This complete the proof of the theorem.
\end{proof}
\eop
\medbreak
\noindent{\bf Example 4.1} The Riordan matrix $(1/(1-kt), t/(1-kt))$ begins 

\[
\left[\begin{array}{cccccccc}
1 &  &  &  &  &  &  &\\ 
k & 1 &  &  &  &  &  &\\ 
k^2 & 2k & 1 &  &  &  &  &\\ 
k^3 & 3k^2 & 3k& 1 &  &  & &\cdots  \\ 
k^4 &4k^3 & 6k^2 & 4k& 1 &  & & \\ 
k^5 & 5 k^4& 10k^3 & 10k^2& 5k & 1 &&  \\ 
&  &  & \cdots  &  &  & &
\end{array}
\right] 
\]

Then its $A$-, $Z$-, and $B$-sequences are $(1,k,0,\ldots)$, $(k,0,\ldots)$, and $(k,0,\ldots)$, respectively, where 
$B$-sequence is defined for all entries of the matrix.
\medbreak

Let $(g,f)$ be a Riordan matrix, where $g=\sum_{j\geq 0} g_j t^j$ and $f=\sum_{j\geq 1}f_jt^j$ with $g_0=1$ and $f_1\not=0$. Equation \eqref{eq:1.2} shows that the generating function $A(t)$ of the $A$-sequence, $(a_j)_{j=0,1,\ldots}$, of $(g,f)$ satisfies 

\[
f(t)=t(a_0+a_1 f+a_2f^2+\cdots),
\]
which implies 

\be\label{6.8}
f_1=a_0\quad \mbox{and} \quad f_2=a_1f_1=a_0a_1.
\ee
Equation \eqref{eq:1.4} shows that the generating function $Z(t)$ of the $Z$-sequence, $(z_j)_{j=0,1,\ldots}$, of $(g,f)$,

\[
g(\bar f)=\frac{1}{1-\bar f Z}.
\]
The above equation can be written as 

\[
Z=\frac{g(\bar f)-1}{\bar f g(\bar f)}=\frac{ g_1+g_2(\bar f)+g_3(\bar f)^2+\cdots}{g_0+g_1(\bar f)+g_2(\bar f)^2+\cdots}.
\]
Thus, 

\be\label{6.9}
z_0=Z(0)=g_1/g_0=g_1,
\ee
and by noting $\bar f'(0)=1/f'(0)=1/f_1$, 

\be\label{6.10}
z_1=Z'(0)=\frac{g_0(g_2\bar f'(0))-g_1^2\bar f'(0)}{g_0^2}=\frac{g_0g_2-g_1^2}{f_1g_0^2}.
\ee
And we have the following result.

\begin{corollary}\label{cor:6.5}
Let $(g,f)$ be a Riordan matrix, where $g=\sum_{j\geq 0} g_j z^j$ and $f=\sum_{j\geq 1}f_jt^j$ with $g_0=1$ and $f_1\not=0$. Then $(g,f)\in R_{1,1,1}$ if and only if $f_1=1$, $f_2=g_1$, and $g_2=g_1^2$, or equivalently, 
$f_1=1$, $f_2=g_1$, and $g_2=f_2^2$.
\end{corollary}

\begin{proof}
One may use \eqref{6.8} - \eqref{6.10} to transfer the sufficient and necessary conditions, $a_0=1$, $z_0=a_1$ and $z_1=0$, for $(g,f)\in R_{1,1,1}$ to be 

\[
f_1=1,\quad g_1=f_2,\quad \mbox{and}\quad g_0g_2=g_1^2,
\]
which proves the corollary.
\end{proof}

\section{$A$- and $B$- sequences of Pascal-like Riordan matrices} 

We shall call a lower-triangular matrix $(a_{n,k})$ is {\it Pascal-like} if 

1.  $a_{n,k}=a_{n,n-k}$ and 

2. $a_{n,o}=a_{n,n}=1$.

It is clear that not all Pascal-like matrices are Riordan matrices. If a Pascal-like matrix is also a Riordan matrix, for example, the Pascal matrix, then it is called a Pascal-like Riordan matrix. 

A Pascal-like matrix will then be the coefficient matrix of a family of monic reciprocal polynomials. Here, a polynomial $P_{n}(x)=\sum^{n}_{k=0}a_{n,k}x^{k}$ of degree $n$ is said to be reciprocal 
if 

\[
P_{n}(x)=x^{n}P_{n}(1/x). 
\]
Hence, we have  

\[
[x^k]P_n(x)=[x^k]x^nP_n(1/x)=[x^k]\sum^n_{j=0}a_{n,j}x^{n-j},
\]
which implies 

\[
a_{n,k}=a_{n,n-k}.
\]

\begin{theorem}\label{thm:4.1}
Let $(a_0,a_1,a_2,\ldots)$ be the $A$-sequence of a Pascal-like Riordan matrix $P=(p_{n,k})_{n\geq k\geq 0}$. Then 

\be\label{4.1}
\left. a_1(1-a_1)\right| a_j
\ee
for $j\geq 2$, or equivalently, $a_2|a_j$ for all $j\geq 2$ due to $a_2=a_1(1-a_1)$. Furthermore, we have recursive formula for $a_j$ as 

\be\label{4.2}
a_{j}=(j-2)a_1(1-a_1)-a_2p_{j-1,2}-\cdots-a_{j-1}p_{j-1,j-2}.
\ee
\end{theorem}

\begin{proof}
We prove \eqref{4.1} by using induction. Let $P$ has $A$-sequence $(a_0,a_1,a_2,\ldots)$. Then, it is easy to see that $a_0=1$ and 

\be\label{4.3}
p_{n,n-1}=1+(n-1)a_1
\ee
for $n\geq 1$. More precisely, we have $p_{1,0}=1$ and $p_{2,1}=1+a_1p_{1,1}=1+a_1$. If $p_{n-1,n-2}=1+(n-2)a_1$, then 

\[
p_{n,n-1}=p_{n-1,n-2}+a_1p_{n-1,n-1}=1+(n-1)a_1.
\]
Since 

\[
p_{3,1}=1+a_1p_{2,1}+a_2
\]
and $p_{3,1}=p_{3,2}$, from \eqref{4.3} we have 

\[
1+a_1p_{2,1}+a_2=1+2a_1,
\]
or equivalently,

\[
a_1(1+a_1)+a_2=2a_1,
\]
which shows that $a_2=a_1(1-a_1).$ 

Assume that $a_1(1-a_1)| a_j$ for all $3\leq j\leq k$, then from the definition of Pascal-like Riordan matrix: $p_{k+1,1}=p_{k+1,k}$, we have

\[
p_{k+1,1}=1+a_1p_{k,1}+a_2p_{k,2}+\cdots +a_kp_{k,k-1}+a_{k+1}=1+ka_1=p_{k+1,k}.
\]
Thus,

\bns
a_{k+1}&=&ka_1-a_1p_{k,1}-a_2p_{k,2}-\cdots-a_kp_{k,k-1}\\
&=&ka_1-a_1p_{k,k-1}-a_2p_{k,2}-\cdots-a_kp_{k,k-1}\\
&=&ka_1-a_1(1+(k-1)a_1)-a_2p_{k,2}-\cdots-a_kp_{k,k-1}\\
&=&(k-1)a_1(1-a_1)-a_2p_{k,2}-\cdots-a_kp_{k,k-1},
\ens
which implies $a_1(1-a_1)|a_{k+1}$ from the induction assumption. The rightmost expression of the above equations imply \eqref{4.2}. Since $a_2=a_1(1-a_1)$, we have $a_2|a_j$ for all $j\geq 2$. 
\end{proof}

\begin{corollary}\label{cor:4.2}
All Pascal-like Riordan matrices have no $B$-sequence except the matrix $(1/(1-t),t)$ and Pascal matrix $(1/(1-t), t/(1-t))$. Here,
the type-I $B$-sequence of $(1/(1-t),t)$ is $(0,0,0,\ldots)$ while its type-II $B$-sequence is $(1,0,0,\ldots)$. Both type-I and type-II $B$-sequences of $(1/(1-t), t/(1-t))$ are $(1,0,0,\ldots)$.
\end{corollary}

\begin{proof}
A Pascal-like matrix has a $B$-sequence if and only if its $A$-sequence possesses sequence element $a_2=a_1(1-a_1)=0$, i.e., $a_1=0$ or $a_1=1$, or equivalently, $P=(1/(1-t),t)$ or $P=(1/(1-t), t/(1-t))$.
\end{proof}

\noindent{\bf Acknowledgements}

We thank the referees and the editor for their careful reading and helpful suggestions.

\end{document}